\documentclass[]{interact}

\usepackage{fancyhdr}
\usepackage{latexsym, epsfig, amsmath, amssymb, amsthm, mathrsfs}
\usepackage{enumerate, xspace, graphics}
\usepackage{enumitem,cite}
\usepackage[bookmarksnumbered, plainpages, backref, colorlinks]{hyperref}
\usepackage{graphicx}

\usepackage{tkz-euclide}
\usepackage{pgfplots}
\usepackage{tikz}
\usepackage{float} 
\usepackage{caption}

\usetikzlibrary{
	plotmarks,patterns
	,decorations.shapes
	,decorations.text
	,decorations.footprints
	,decorations.fractals
	,decorations.pathmorphing
	,shadows
	,matrix
	,positioning
	,fadings
}

\usepackage[longnamesfirst,sort]{natbib}
\usepackage{epstopdf}
\usepackage[caption=false]{subfig}
\bibpunct[, ]{[}{]}{,}{n}{,}{,}
\makeatletter
\def\NAT@def@citea{\def\@citea{\NAT@separator}}
\makeatother

\theoremstyle{plain}
\newtheorem{theorem}{Theorem}[section]

\newtheorem{proposition}[theorem]{Proposition}
\newtheorem{algorithm}[theorem]{Algorithm} 

\theoremstyle{definition}
\newtheorem{definition}[theorem]{Definition}
\newtheorem{example}[theorem]{Example}

\theoremstyle{remark}
\newtheorem{remark}{Remark}

\newcommand{\ds}{\displaystyle}

\newcommand{\R}{\mathbb{R}}

\def\int{{\textrm{int}}}

\DeclareMathOperator*{\argmin}{argmin}

\begin{document}
	\articletype{ARTICLE TEMPLATE}

\title{A Scaled Gradient Modified Non-monotone Line Search Method for Constrained Optimization Problems\thanks{The
research work of the first and second author was done during their visit to
Center for General Education, China Medical University, Taichung, Taiwan, and
Department of Applied Mathematics, National Sun Yat-sen University, Kaohsiung, Taiwan.}}

\author{
\name{Qamrul Hasan Ansari\textsuperscript{a}\thanks{Email: qhansari@gmail.com},
Feeroz babu\textsuperscript{b}\thanks{Email: firoz77b@gmail.com, \, feerozbabu@vitbhopal.ac.in},
D.~R. Sahu\textsuperscript{c}\thanks{Email: drsahudr@gmail.com}
Jen-Chih Yao\textsuperscript{d}\thanks{Email: yaojc@mail.cmu.edu.tw}}
\affil{\textsuperscript{a}Department of Mathematics,
King Fahd University of Petroleum \& Minerals, Dhahran, Saudi Arabia;\\
\textsuperscript{b}Mathematics Division, School of Advanced Sciences \& Languages,
VIT Bhopal University, Bhopal-Indore Highway, Kothrikalan, Sehore, Madhya Pradesh 466 114, India;\\
\textsuperscript{c}Department of Mathematics, Banaras Hindu University, Varanasi, India;\\
\textsuperscript{d}Center for General Education, China Medical University, Taichung 40402, Taiwan}}

\maketitle

\begin{abstract}
In this paper, we propose a scaled gradient modified non-monotone line search method
for solving constrained minimization problems,
and explore several specific properties of this method, namely, its convergence analysis.
We discuss the linear convergence rate of the sequence generated by the proposed algorithm
to a solution of the constrained minimization problem where the objective function
is strongly quasiconvex.
We consider numerical examples of large-scale fractional programming and
quadratic programming for the function of pseudo convex and strongly quasiconvex
and compare the performance of the proposed algorithm with the existing ones for these examples.
\end{abstract}


\begin{amscode}
90C06; 49M37; 90C26; 65Y20
\end{amscode}

	
\section{Introduction}\label{S1}
	

The conjugate gradient method, $x_{k+1} = x_k +\lambda_k d_k$,
where $d_k$ is a descent direction and $\lambda_k > 0$ is a step size,
is one of the most popular iterative methods in the realm of the following unconstrained optimization problem:
\begin{equation}\label{Eq:NCOP}
\min_{x \in \mathbb{R}^{n}} f(x),
\end{equation}
where $f : \mathbb{R}^{n} \to \mathbb{R}$ is a continuously differentiable function.
In this method, the step size $\lambda_{k}$ plays a fundamental role,
particularly, in monotone line search methods where the step size $\lambda_k$
is meticulously chosen to guarantee a decrease in the objective function,
i.e.,  $f(x_{k+1}) < f(x_k)$.
However, non-monotone line search methods relax this strict monotonicity requirement, and
allow an inevitable function value growth.
Many researchers (see, for example, \cite{GLL86,BMR00,D02,ZH04,HWC14,SS06}) have emphasized
the potential of non-monotone schemes to enhance the likelihood of discovering
a global optimum and improve convergence speed,
especially in scenarios where a monotone scheme encounters challenges navigating narrow curved valleys.

The origin of the non-monotone line search framework can be traced back to
the pioneer work of Grippo et al. \cite{GLL86}, primarily in the context of Newton's method
for problem \eqref{Eq:NCOP},
in which step size $\lambda_{k}$ satisfies the condition
\begin{equation}\label{Eq:1}
	f(x_{k} + \lambda_{k} d_{k}) \leq \max_{0 \leq j \leq m_{k}} f(x_{k-j})
	+ \delta \lambda_{k} \nabla f(x_{k})^{\top} d_{k},
\end{equation}
where $d_k$ is descent direction, $\delta \in (0,1)$, $m_{0} = 0$, $0 \leq m_{k} \leq \min \{ m_{k-1}+1, M \}$
and $M$ is a prefixed nonnegative integer.
Their approach allows some increase in function value in iterations,
deviating from the strictly decreasing nature of monotone approaches.
However, certain drawbacks have been identified despite the success of these
non-monotone techniques based on the mentioned principles.
Dai \cite{D02} pointed out that even when an iterative method generates
R-linearly convergent sequence for a strongly convex function,
the iterations might not adhere to the conditions specified in the non-monotone approach for sufficiently large $k$.
Moreover, Toint \cite{T96} emphasized the potential benefits of non-monotone schemes,
stating their ability to improve the likelihood of finding a global optimum and enhancing convergence speed.
This assertion is corroborated by numerical verifications in the literature.

Shi and Shen \cite{SS06} and Zhang and Hager \cite{ZH04} contributed to this area
by developing and analyzing non-monotone line search algorithms with superior performance
in function and gradient evaluations.
These studies underscore the significance of non-monotone line search methods
in achieving faster and better convergence in convex minimization.
Hu et al. \cite{HHL10} improved the non-monotone line search technique,
which provides a global optimal solution and better numerical performance.
Further, Nesterov and Shikman \cite{NS14} and Huang et al. \cite{HWC14} proposed new methods
that guarantee the best possible rate of convergence.

In particular, Zhang and Hager \cite{ZH04} replaced the maximum function value from \eqref{Eq:1} by an average
of function values. Gu and Mo \cite{GM08} further replaced the average of function values by the
convex combination of the previous non-monotone term and the current objective function value.
Ou and Liu \cite{OL17} adopted this approach, but also included a term $\delta_{2} \lambda^{2} \| d_{k} \|^{2}$
in the non-monotone line research. Their modified non-monotone line search is stated as follows:
Given $\delta_{1} \in (0,1)$, $\delta_{2} > 0$ and $\beta \in (0,1)$.
Set $s_{k} = -\frac{\nabla f(x_{k})^{\top} d_{k}}{ \| d_{k} \|^{2} }$,
and select step size $\lambda_{k}=s_k\beta^{j_k}$, where $j_k$ is the smallest integer such that
\begin{equation}\label{Eq:2}
f(x_{k} + \lambda_k d_{k}) \leq \mathcal{T}_{k} + \delta_{1} \lambda_k \nabla f(x_{k})^{\top} d_{k}
- \delta_{2} \lambda^{2}_k \| d_{k} \|^{2},
\end{equation}
where $\mathcal{T}_{0} = f(x_{0})$, $\mathcal{T}_{k} = \eta_{k} \mathcal{T}_{k-1} + (1-\eta_{k}) f(x_{k})$
for $k \geq 1$ with $0 \leq \eta_{k} \leq \eta_{max} < 1$, and $d_{k}$ satisfies
$\nabla f(x_{k})^{\top} d_{k} < 0$

In this paper, we consider the following constrained minimization problem:
\begin{eqnarray}\label{P1}
\min_{x\in K} f(x),
\end{eqnarray}
where $K$ is a nonempty closed and convex subset of $\mathbb{R}^n$
and $f : K \to \mathbb{R}$ is a continuously differentiable function.
We denote by $\Omega_{\min}(f,K)$ the set of all solutions of the minimization problem \eqref{P1}.
Numerous fields, including neurology, network/traffic problems, circuit design problems,
machine learning and artificial intelligence, signal and image processing,
radiation therapy treatment planning, astronomy, data science, and many others,
have been benefited from the use of the mathematical formulation of the minimization model \eqref{P1},
see, \cite{BB88,SZ05,BN06,BLZ08,P69,PR69}.

For constrained minimization problem \eqref{P1},
Bonettini et al. \cite{ZH09} introduced the scaled gradient projection (in short, SGP) method
with the non-monotone line search procedure \eqref{Eq:1} by incorporated
a scaling matrix $\mathcal{D}_k$ with gradient $\nabla f(x_k)$.
Bonettini and Prato \cite{BP15} further studied the convergence analysis of the scaled
gradient projection method and established $\mathcal{O}(1/k)$ convergence rate
of the sequence generated by the SGP method under the condition that the gradient of the
objective function is Lipschitz continuous.
It is worth to mention that the implementation of scaling matrix $\mathcal{D}_k$
provides a better rate of convergence, and hence, it is an advancement from a computation viewpoint.
Yan et al. \cite{YWH18} considered the SGP method with non-monotone line research procedure
proposed by Zhang and Hager \cite{ZH04}, and studied its $\mathcal{O}(1/k)$ convergence rate
under the condition that the objective function is strongly  convex.
Further, the linear convergence of the SGP method is also derived under the strong convexity assumption of the
involved objective function.

In this study, we propose a scaled gradient method with a modified line search method \eqref{Eq:2}.
Moreover, the convergence analysis of the proposed algorithm is presented, and
we demonstrate a linear convergence rate for a strongly quasiconvex objective function,
which is better than $\mathcal{O}(1/k)$ rate of convergence studied by Yan et al. \cite{YWH18}.

The subsequent sections of this paper are organized as follows:
Section \ref{S:2} consists of notations, basic definitions and preliminary results which are required for this paper.
Section \ref{S:3} proposes a scaled gradient modified non-monotone line search method involving
scaling matrix $\mathcal{D}_k$ for solving constrained optimization problems.
Several specific properties of the proposed method are also explored, namely, convergence analysis.
In Section \ref{S:4}, the linear convergence of the sequence generated by the proposed algorithm is studied
under the strong quasiconvexity assumption.
In Section \ref{S:5}, we consider numerical examples of large-scale quadratic programming and fractional programming for pseudo convex and strongly quasiconvex functions, and compare the performance of the proposed algorithm with the existing methods of
Zhang and Hager \cite{ZH04} and Yan et al. \cite{YWH18}.

	
\section{Preliminaries}\label{S:2}
	

Throughout the paper, we adopt the following terminology and notations.
Let $\mathbb{R}^n$ be the $n$-dimensional Euclidean space whose norm is denoted by $\| \cdot \|$.
For $x\in \mathbb{R}^n$, $\|x\|_\mathcal{D} := \sqrt{x^{\top}\mathcal{D}x}$ denotes
the $\mathcal{D}$-norm of $x$ induced by the symmetric positive definite matrix $\mathcal{D}$.
The transpose operation is denoted by the superscript $^{\top}$.
For a given $\mu \geq 1$, let $\mathfrak{D}_\mu$ be the set of symmetric positive definite matrices
with all eigenvalues distributed in the interval $[ \frac{1}{\mu}, \mu ]$.
Consequently, for any matrix $\mathcal{D}\in \mathfrak{D}_{\mu}$,
it is evident that $\mathcal{D}^{-1}$ also belongs to $\mathfrak{D}_{\mu}$
such that $\|\mathcal{D}\|\leq \mu$ and $\|\mathcal{D}^{-1}\|\leq \mu$, and furthermore,
\begin{equation}\label{eq1.30}
\frac{1}{\mu}\|x\|^2 \leq \|x\|_{\mathcal{D}}^2 \leq \mu\|x\|^2,\quad \forall x \in \mathbb{R}^n.
\end{equation}
For more details, we refer \cite{BP15,FP03,ZH09}.
	
This paper uses $K$ as a nonempty closed and convex subset of $\mathbb{R}^n$.
The metric projection of $\mathbb{R}^n$ onto $K$ equipped with the $\mathcal{D}-$norm,
denoted by $P_{K, \mathcal{D}}$, is defined by
$$ P_{K, \mathcal{D}}(x) = \argmin_{y \in K} \left\{\|y-x\|_\mathcal{D}\right\}
= \argmin_{y \in K} \left\{ \frac{1}{2}y^{\top}\mathcal{D}y - y^{\top}\mathcal{D}x\right\}.$$
As mentioned in \cite[p. 81]{FP03} that $P_{K, \mathcal{D}}(x)$ is characterized by the property
\begin{align}\label{eq1.29}
[P_{K, \mathcal{D}}(x) - x]^{\top}\mathcal{D}[P_{K, \mathcal{D}}(x) - y] \leq 0,
\quad \forall x\in \mathbb{R}^n, y \in K,
\end{align}
which is further equivalent to
\begin{align}\label{eq1.299}
\| P_{K, \mathcal{D}}(x_1) - P_{K, \mathcal{D}}(x_2)\| \leq \mu^2\| x_1- x_2\|,
\quad \forall x_1, x_2\in \mathbb{R}^n.
\end{align}

\begin{definition}
A function $f: K \to\R$ is said to be
\begin{enumerate}
\item [(a)] convex on $K$ (see \cite{F53}) if
\begin{equation*}
	f(t x+(1-t)y)\leq t f(x)+(1-t)f(y),
	\quad \forall x,y \in K, \, \forall t \in[0,1];
\end{equation*}
\item [(b)] quasiconvex on $K$  if
\begin{equation*}
	f(t x+(1-t)y)\leq \max\{ f(x), f(y)\},
	\quad \forall x,y \in K, \, \forall t \in[0,1];
\end{equation*}
\item [(c)] strongly quasiconvex on $K$ if there exists a constant $\gamma > 0$ such that
\begin{equation*}
	f(t x+(1-t)y) \leq \max\{f(x), f(y)\} - \frac{\gamma}{2}\|x-y\|^2,
	\quad \forall x,y \in K, \, \forall t \in [0,1].
\end{equation*}
\item [(d)]  If $f: K \to\R$ is differentiable, then it is called pseudo-convex (see \cite{ALM14})
if for all $x, y \in K$,
$$ \nabla f(x)^{\top} (y-x) \geq 0 \quad \Rightarrow \quad  f(y) \geq f(x).$$
\end{enumerate}
\end{definition}


	
\section{Non-monotone Line Search Method and Its Convergence Analysis}\label{S:3}
	

In this section, we introduce a non-monotone line search method to solve minimization problem \eqref{P1}
and analyze its convergence behaviour.

Let $K$ be a nonempty closed and convex subset of $\mathbb{R}^n$ and
$f:K \to \mathbb{R}$ be a continuously differentiable function.
Recall that, when $f$ is pseudo-convex, then $x^*\in K$ is a solution of
the following variational inequality problem (in short, VIP):
\begin{equation}\label{Statpoint:eq1}
\mbox{find } x^*\in K \mbox{ such that } \,  -\nabla  f(x^*)^{\top}(x-x^*) \leq 0, \quad \forall x\in K,
\end{equation}
if and only if it is a solution of the minimization problem \eqref{P1} (see \cite{ALM14}).

Moreover, from \cite[Lemma 2.2]{ZH09}, $x^*\in K$ is a solution of VIP \eqref{Statpoint:eq1} if and only if
\begin{equation}\label{Statpoint:eq2}
x^{*} = P_{K, \mathcal{D}^{-1}} \left( x^*-\alpha\mathcal{D}\nabla  f(x^*) \right),
\quad \forall \alpha>0.
\end{equation}

Motivated by \eqref{Statpoint:eq2}, we now introduce a non-monotone line search method
to solve minimization problem \eqref{P1}.

Let $\{ \alpha_k \}$ and $\{\eta_k \}$ be sequences in $(0,\infty)$, $\{ \mu_{k} \}$ be a sequence
in $[1, \infty)$ and $\mu > 1$ such that
\begin{equation}\label{Cond:alphaeta}
0<\alpha_{\min} \leq \alpha_k \leq \alpha_{\max} <\infty, \, \, 1\leq \mu_k \leq \mu \, \,
\textrm{and} \, \, 0 \leq \eta_k   \leq \eta_{\max}< 1, \quad \forall k\ge 0.
\end{equation}
\newline
\\
\noindent \rule{\textwidth}{1pt}
\begin{algorithm}[Scaled gradient modified non-monotone method (SGM)]
	\label{A:1}
	\quad
	\vspace{2pt}
	\hrule
	\vspace{3pt}
	\noindent Choose $x_0\in K$ and parameters $0 < \delta_1 <\beta< 1$, $\mu > 1$ and  $\delta_2 > 0$.
    Set $k=0$.\\
	\noindent{\sc Step 1:} Set $\mathcal{T}_0 := f(x_0)$.\\
	\noindent{\sc Step 2:} Choose the scaling matrix $\mathcal{D}_k \in \mathfrak{D}_{\mu_k}$, then compute
	\begin{equation}\label{eqn1 Alg1}
	y_k = P_{K,\mathcal{D}_{k}^{-1}} \left( x_k - \alpha_k\mathcal{D}_k\nabla f(x_k) \right).
	\end{equation}
	Let $d_k = y_k-x_k$. If $d_k = \mathbf{0}$, then stop.\\
	\noindent{\sc Step 3:} Otherwise, assume $s_k=-\frac{\nabla f(x_k)^{\top}d_k}{\|d_k\|^2}$,
    and select $\lambda_k =s_k\beta^{j_k}$, where $j_k$ is the smallest nonnegative integer
    such that the following Armijo's condition holds
	\begin{equation}\label{eq3}
	f(x_k + \lambda_kd_k) \leq \mathcal{T}_k + \delta_1\lambda_k \nabla f(x_k)^{\top}d_k-\delta_2\lambda_{k}^2\|d_k\|^2.
	\end{equation}
	\noindent  Compute
	\begin{equation}\label{eq6}
	x_{k+1} = x_k  +\lambda_kd_k.
	\end{equation}
	Choose
	\begin{equation}\label{eq7}
	\mathcal{T}_{k+1}= \eta_{k+1} \mathcal{T}_{k} + (1-\eta_{k+1}) f(x_{k+1}), \quad k\geq 0.
	\end{equation}
	Set $k := k+1$ and go to {\sc Step 1}.
	\vspace{5pt}
	\hrule
\end{algorithm}

\begin{remark}
\begin{itemize}
\item[(a)] If $d_k = \mathbf{0}$ or $y_k = x_k$, then it follows directly from \eqref{Statpoint:eq2} that
$x_k$ is a solution of VIP \eqref{Statpoint:eq1}.

\item[(b)] We  reiterate that the modified non-monotone line search approach from
step \eqref{eqn1 Alg1} - \eqref{eq7} is what distinguishes Algorithm \ref{A:1}
from the non-monotone line search methods studied in \cite{ZH04,ZH09,YWH18,OL17}.
For example, Zhang and Hager \cite{ZH04} considered an upper bound condition on
the sequence $\{\lambda_{k}\}$ in the Armijo's scheme,
which is not required in Algorithm \ref{A:1}.
By including a scaling matrix $\mathcal{D}_k$,
the performance of Algorithm \ref{A:1} is improved as demonstrated in Example \ref{ex1}.
It has also been shown that \eqref{eq3} and \eqref{eq7} provide stability in computation.

\item[(c)] Note that the non-monotone line search \eqref{eq3} transforms
to the non-monotone line search method in \cite{YWH18}
if $\delta_2 = 0$ and $s_k = s$ for all $k\geq 0$.
For all $k\geq 0$ and $\mathcal{D}_k=1$, the non-monotone line search \eqref{eq3} reduces
to the monotone line search method in \cite{ZH09} if $\delta_2 = 0$, $s_k = s$ and $\eta_{k}=0$.
Therefore, the line search strategy \eqref{eq3} can be considered as a modified version
of the line search methods in \cite{ZH09,YWH18}.
\end{itemize}
\end{remark}

In order to achieve meaningful convergence outcome from Algorithm \ref{A:1},
we consider the following assumptions:
\begin{itemize}	
\item[] {\rm (A1)} The level set $\mathcal{L} :=\{x\in K: f(x)\leq f(x_0)\}$ is bounded,
where $x_0$ is the initial point in Algorithm \ref{A:1}.

\item[] {\rm (A2)} The solution set $\Omega_{\min}(f,K)$ of the minimization problem \eqref{P1} is nonempty.

\item[] {\rm (A3)} The gradient $\nabla f$ of $f$ is $L$-Lipschitz continuous on $K$, i.e.,
there exists $L > 0$ such that $\|\nabla f(x)-\nabla f(y)\|\leq L\|x-y\|$ for all $x,y \in K$.
\end{itemize}

The following example shows that the family of functions which satisfy the assumption (A1)--(A3)
is nonempty.

\begin{example}
Consider the convex and compact set
$K  = \{ x\in \mathbb{R} : -1 \leq x \leq 1 \}$, and define a function $ f : K \to \mathbb{R} $  by
$$ f(x) = \frac{10x^2}{1+ 10x^2}.$$
Then, $f$ is pseudo-convex, but not convex on $K$, also satisfies assumptions (A1)--(A3).
Moreover, the function $f$ satisfies the descent condition. Indeed, let $ d_k := -\nabla f(x_k)$. Then,
$\nabla f(x_k)^\top d_k = -\|\nabla f(x_k)\|^2 \leq -c \|\nabla f(x_k)\|^2$ for any $c \in (0,1]$,
in particular, for $c = 1$. Hence, the descent condition
$\nabla f(x_k)^\top d_k \leq -c \|\nabla f(x_k)\|^2$, for all $k \geq 0$
is satisfied.
\end{example}


\begin{remark}\label{rm2}
		If $f:\mathbb{R}^n\to \mathbb{R}$ is strongly quasiconvex with constant $\gamma > 0$, then the following statements hold:
		\begin{enumerate}
			\item[(a)] {\rm \cite[Theorem 1.25]{ALM14}} The level set $\mathcal{L}$ is convex.
			\item[(b)] {\rm\cite[Theorem 2]{K89}} If \(f:\mathbb{R}^n\to\mathbb{R}\) is differentiable with \(L\)-Lipschitz gradient  then
			$$
			f(x) - f(x^{*}) \leq \frac{L}{\gamma^2} \| \nabla f(x) \|^{2},\quad  \forall x \in \mathbb{R}^{n}.
			$$
		\end{enumerate}
\end{remark}

We present the following result to show that the direction $d_k$ defined in Algorithm \ref{A:1}
is a descent direction,
that is, $x_{k} + d_{k} \in K$ and $f(x_{k})^{\top} d_{k} < 0$.

\begin{proposition}\label{lm1}
Assume that $\{x_k\}$, $\{y_k\}$ and $\{d_k\}$ are sequences generated by Algorithm {\rm \ref{A:1}}.
Then, the following assertions hold:
\begin{enumerate}
\item [{\rm(a)}] For $d_k\neq {\bold 0}$,
\begin{equation}\label{a1}
\nabla  f(x_k)^{\top}d_k \leq -\frac{\|d_k\|_{\mathcal{D}_k^{-1}}^2}{\alpha_k}
\leq-\frac{\|d_k\|_{\mathcal{D}_k^{-1}}^2}{\alpha_{\max}} <0,
\end{equation}
that is, $d_k$ is a descent direction for the minimization problem \eqref{P1} at the point $x_k$;
and consequently,
$$s_{k} = -\frac{\nabla f(x_{k})^{\top }d_{k}}{\Vert d_{k}\Vert^{2}}\geq \frac{1}{\mu\alpha_{\max}}.$$

\item [{\rm(b)}] For all $k\geq 0$, we have
\begin{equation}\label{a2}
\|d_k\| \leq \mu^3\alpha_{\max} \|\nabla f(x_k)\|.
\end{equation}
\end{enumerate}
\end{proposition}

\begin{proof}
(a) Although, it is similar to the one in \cite{YWH18}, we include it for the sake of completeness of the paper.

Suppose that $d_k \neq {\bold 0}$.
Since $y_k = P_{K,\mathcal{D}_{k}^{-1}}(x_k - \alpha_k\mathcal{D}_k\nabla f(x_k))$,
by the characterization property \eqref{eq1.29} of the metric projection and in view of the fact that $x_{k} \in K$,
we observe that
$$ [y_k - x_k + \alpha_k\mathcal{D}_k\nabla f(x_k)]^{\top} \mathcal{D}_{k}^{-1}[y_k - x_k]  \leq 0,$$
that is,
$$ [y_k - x_k ]^{\top} \mathcal{D}_{k}^{-1}[y_k - x_k] + \alpha_k\nabla f(x_k)^{\top}[y_k - x_k]  \leq 0.$$
Thus,
$$ \|y_k - x_k\|_{\mathcal{D}_{k}^{-1}}^2 + \alpha_k\nabla f(x_k)^{\top}d_k \leq 0,$$
which by \eqref{Cond:alphaeta} implies that
$$ \nabla f(x_k)^{\top}d_k \leq -\frac{\|d_k\|_{\mathcal{D}_k^{-1}}^2}{\alpha_k}
\leq-\frac{\|d_k\|_{\mathcal{D}_k^{-1}}^2}{\alpha_{\max}} <0.$$
From \eqref{eq1.30}, we have
\begin{equation*}
-\Vert d_{k}\Vert _{\mathcal{D}_{k}^{-1}}^{2}\leq -\frac{1}{\mu }\Vert d_{k}\Vert ^{2}.
\end{equation*}
Hence, from (\ref{a1}), we obtain
\begin{equation*}
\nabla f(x_{k})^{\top }d_{k}\leq -\frac{1}{\alpha _{\max }}\Vert d_{k}\Vert_{\mathcal{D}_{k}^{-1}}^{2}
\leq -\frac{1}{\mu\alpha _{\max } }\Vert d_{k}\Vert ^{2},
\end{equation*}%
and thus, $s_{k}$ defined as in {\sc Step 3} of Algorithm \ref{A:1} is such that
\begin{equation*}
s_{k} = -\frac{\nabla f(x_{k})^{\top }d_{k}}{\Vert d_{k}\Vert^{2}}\geq \frac{1}{\mu\alpha _{\max}}.
\end{equation*}

(b) By invoking \eqref{eq1.299} and the relation
$y_k = P_{K,\mathcal{D}_{k}^{-1}}(x_k - \alpha_k\mathcal{D}_k\nabla f(x_k))$,
we obtain
\begin{align*}
\|d_k\|=\|y_k- x_k\|
&= \|P_{K,\mathcal{D}_{k}^{-1}}(x_k - \alpha_k\mathcal{D}_{k}\nabla f(x_k)) - x_k\|\\
&\leq \|P_{K,\mathcal{D}_{k}^{-1}}(x_k - \alpha_k\mathcal{D}_{k}\nabla f(x_k)) - P_{K,\mathcal{D}_{k}^{-1}}(x_k)\|\\
& \leq \mu^2\|x_k - \alpha_k\mathcal{D}_k\nabla f(x_k) - x_k\|\\
& = \mu^2\alpha_{k} \|\mathcal{D}_k\nabla f(x_k)\|\\
&\leq \mu^2\alpha_{\max}\|\mathcal{D}_k\nabla f(x_k) \|\\
&\leq \mu^3\alpha_{\max}\|\nabla f(x_k) \|.
\end{align*}
Therefore,
$$ \|d_k\| \leq \mu^3\alpha_{\max}\|\nabla f(x_k)\|,$$
and hence, the proof is complete.
\qed
\end{proof}
		
The following proposition provides a necessary and sufficient condition for an accumulation point of the sequence
produced by Algorithm \ref{A:1} to be a solution of VIP \eqref{Statpoint:eq1}.

\begin{proposition}{\rm \cite[Proposition 2.5]{ZH09}}\label{lm7}
Suppose that a subsequence $\{x_{k_j}\}$ converging to a point $x^*\in K$.
Then, the accumulation point $x^*$ is a solution of VIP \eqref{Statpoint:eq1} if and only if	
$$ \lim\limits_{j\to \infty} \nabla f(x_{k_j})^{\top}d_{k_j} =0.$$
\end{proposition}

\begin{proposition}\label{lm2}\label{lm5}
Let $\{x_k\}$ be a sequence generated by Algorithm \ref{A:1}. Then, the following assertions hold:
\begin{itemize}
  \item [{\rm(a)}] $f(x_{k+1}) \leq \mathcal{T}_k$   for all $k \geq 0.$
  \item [{\rm(b)}]  $f(x_{k})\le\mathcal{T}_k$  for all $k \geq 0.$
  \item [{\rm(c)}]  The sequence $\{\mathcal{T}_{k}\}_{k\geq 0}$ is monotonically non-increasing.
  \item [{\rm(d)}]  $x_k\in \mathcal{L}$ for all $k \geq 0$.
\end{itemize}

\end{proposition}
\begin{proof}
(a) Note that, in view of \eqref{eq3}, we get
$$ f(x_{k+1}) \leq \mathcal{T}_k + \delta_1\lambda_k \nabla f(x_k)^{\top}d_k-\delta_2\lambda_{k}^2\|d_k\|^2
\leq \mathcal{T}_k, \quad \forall k \geq 0,$$
where the second inequality is due to \eqref{a1}.

(b), (c) and (d) follows from \cite{OL17}.
\qed
\end{proof}

\begin{proposition}\label{lm4}{\rm \cite{OL17}}
The Algorithm \ref{A:1} is well defined, i.e., the Armijo's line search \eqref{eq3} is satisfied.
\end{proposition}

\begin{proof}
It lies on the lines on the proof of Lemma 3.2 in \cite{OL17}, and therefore, it is omitted.
\end{proof}

\begin{proposition}\label{lm3}
Let $K$ be a nonempty closed and convex subset of $\mathbb{R}^n$,
$f : K\to \mathbb{R}$ be a differentiable function such that the assumptions (A1) and (A3) hold.
Let $\{x_k\}$ be a sequence generated by Algorithm {\rm\ref{A:1}}.
Then, there exists a constant $0< \zeta<1$ such that
\begin{equation}\label{Eq1H1H2}
\mathcal{T}_k - f(x_{k+1})  \geq \zeta \left(\frac{\nabla f(x_k)^{\top}d_k}{\|d_k\|} \right)^2,
\quad \forall k \ge 0.
\end{equation}
\end{proposition}

\begin{proof}
It lies on the lines of the proof of Lemma 3.5 in \cite{OL17}.
However, we include it for the sake of readers.

From Proposition \ref{lm4}, Armijo's line search \eqref{eq3} is satisfied.
We define two sets $H_1 := \{ k \geq 0 : \lambda_{k} = s_k\}$ and
$H_2 := \{ k \geq 0 : \lambda_{k} < s_k\}$, and consider the following two cases:\\

\noindent{\sc Case 1:} When $k\in H_1$.

From the Armijo's line search \eqref{eq3} and $s_k=\frac{-\nabla f(x_k)^{\top}d_k}{\|d_k\|^2}$, we have
\begin{align*}
	\mathcal{T}_k - f(x_{k+1})
	&\geq -\delta_1\lambda_k \nabla f(x_k)^{\top}d_k+\delta_2\lambda_{k}^2\|d_k\|^2\\
	&=-\delta_1s_k \nabla f(x_k)^{\top}d_k+\delta_2s_{k}^2\|d_k\|^2\\
	&=\delta_1\frac{\nabla f(x_k)^{\top}d_k}{\|d_k\|^2} \nabla f(x_k)^{\top}d_k
    +\delta_2 \left( \frac{\nabla f(x_k)^{\top}d_k}{\|d_k\|^2} \right)^2\|d_k\|^2\\
	&=\delta_1\left(\frac{\nabla f(x_k)^{\top}d_k}{\|d_k\|}\right)^2
    +\delta_2\left(\frac{\nabla f(x_k)^{\top}d_k}{\|d_k\|^2}\right)^2\\
	& = (\delta_1 + \delta_2)\left(\frac{\nabla f(x_k)^{\top}d_k}{\|d_k\|} \right)^2.
\end{align*}
Thus, in this case, the inequality (\ref{Eq1H1H2}) holds.\\

\noindent{\sc Case 2:} When $k\in H_2$.

Since $\lambda_{k} < s_k$, we have $\lambda_{k} = s_k\beta^{j_k}$ for some $j_k > 0$.
Here, $j_k$ is the smallest positive integer such that Armijo's line search \eqref{eq3} is satisfied.
Now, set $\lambda := \frac{\lambda_k}{\beta}$, that is, $\lambda = \lambda_k\beta^{-1} = s_k\beta^{j_k-1}$.
Then, due to the choice of $\lambda$, it does not satisfy the Armijo's line search \eqref{eq3}.
Hence,
$$ f(x_{k} +\lambda d_k) > \mathcal{T}_k + \delta_1\lambda \nabla f(x_k)^{\top}d_k-\delta_2\lambda^2\|d_k\|^2.$$
By Proposition \ref{lm2}(b), it yields that
\begin{equation}\label{32}
 f(x_{k} +\lambda d_k) > f(x_k) + \delta_1\lambda \nabla f(x_k)^{\top}d_k-\delta_2\lambda^2\|d_k\|^2.
\end{equation}
By mean-value theorem, that there exists $\xi_k \in (0,1)$ such that
$$ f(x_{k} +\lambda d_k) = f(x_k) + \lambda \nabla f(x_{k} +\lambda \xi_kd_k)^{\top} d_k.$$
This together with the inequality \eqref{32} implies that
$$ \lambda \nabla f(x_{k} +\lambda \xi_kd_k)^{\top} d_k >
\delta_1\lambda \nabla f(x_k)^{\top}d_k-\delta_2\lambda^2\|d_k\|^2,$$
that is,
\begin{align}\label{eq1.50}
\nabla f(x_{k} +\lambda \xi_kd_k)^{\top} d_k >  \delta_1 \nabla f(x_k)^{\top}d_k-\delta_2\lambda\|d_k\|^2.
\end{align}
On the other hand, by the Lipschitz continuity of $\nabla f$, \eqref{eq1.50} and the Cauchy-Schwarz inequality,
we observe that
\begin{align*}
	L\lambda\|d_k\|^2 > L\lambda\xi_k\|d_k\|^2& \geq \|\nabla f(x_{k} +\lambda \xi_kd_k) - \nabla f(x_k)\| \|d_k\|\\
	& \geq (\nabla f(x_{k} +\lambda \xi_kd_k) - \nabla f(x_k))^{\top}d_k\\
	&=\nabla f(x_{k} +\lambda \xi_kd_k)^{\top}d_k - \nabla f(x_k)^{\top}d_k\\
	& > -(1-\delta_1)\nabla f(x_k)^{\top}d_k - \delta_2\lambda\|d_k\|^2,
\end{align*}
and therefore,
$$ \lambda_{k} = \lambda\beta >
\frac{\beta(1-\delta_1)}{L+\delta_2}\left(\frac{\nabla f(x_k)^{\top}d_k}{\|d_k\|^2} \right)
=\frac{\beta(1-\delta_1)}{L+\delta_2}s_k, \quad \forall k \in H_2.$$
Then from Armijo's line search \eqref{eq3}, we have
\begin{align*}
\mathcal{T}_k - f(x_{k+1})
&\geq -\delta_1\lambda_k \nabla f(x_k)^{\top}d_k + \delta_2\lambda_{k}^2\|d_k\|^2\\
&\geq  \frac{\beta\delta_1(1-\delta_1)}{L+\delta_2}\left(\frac{\nabla f(x_k)^{\top}d_k}{\|d_k\|} \right)^2
+  \frac{\beta^2\delta_2(1-\delta_1)^2}{(L+\delta_2)^2}\left(\frac{\nabla f(x_k)^{\top}d_k}{\|d_k\|} \right)^2\\
&= \left(\frac{\beta\delta_1(1-\delta_1)}{L+\delta_2}
+  \frac{\beta^2\delta_2(1-\delta_1)^2}{(L+\delta_2)^2}\right)\left(\frac{\nabla f(x_k)^{\top}d_k}{\|d_k\|} \right)^2\\
&= \zeta\left(\frac{\nabla f(x_k)^{\top}d_k}{\|d_k\|} \right)^2,
\end{align*}
where $\zeta = \min \left\{ \delta_1 + \delta_2, \frac{\beta\delta_1(1-\delta_1)}{L+\delta_2}
+  \frac{\beta^2\delta_2(1-\delta_1)^2}{(L+\delta_2)^2}\right\} < 1$.
From \noindent{\sc Case 1} and 2, we conclude that the inequality (\ref{Eq1H1H2}) holds for all $k\ge0.$
\qed
\end{proof}

\begin{remark}
Zhang and Hager \cite[Theorem 3.1]{ZH04} proved local convergence result
by requiring the following direction assumptions:
There exist positive constants $c$ and $c_1$ such that
\begin{equation}\label{eq:direction_assumption}
	\nabla f(x_k)^{\top}d_k \leq -c\|\nabla f(x_k)\|^2, \quad \forall k\geq 0.
\end{equation}
\begin{equation}
	\|d_k\| \leq  c_1\|\nabla f(x_k)\|, \quad \forall k\geq 0.
\end{equation}
However, to establish global convergence, we rely only on  the direction assumption \eqref{eq:direction_assumption}.
\end{remark}

\begin{proposition}\label{lm6}
Let $f:K \to \mathbb{R}$ be a differentiable function such that
$\nabla f$ is Lipschitz continuous with Lipschitz constant $L$
and assume that the Armijo's condition \eqref{eq3} is satisfied for all $k \geq 0$.
Then the following assertions hold:
\begin{itemize}
	\item [{\rm(a)}]\begin{equation}\label{3.1}
		\lambda_{k} \geq  \frac{(1-\delta_1)}{(1+\delta_2 L/2)\mu\alpha_{\max}} :=\lambda_{\min}.
	\end{equation}
	\item [{\rm (b)}] If there exists $c>0$ such that direction assumption \eqref{eq:direction_assumption} holds, then
	\begin{equation}\label{Eq:3.11}
		\lambda_{k} \leq \frac{1}{c}.
	\end{equation}
\end{itemize}
\end{proposition}

\begin{proof}
(a) Since the Armijo's condition \eqref{eq3} is satisfied, there is a smallest nonnegative integer $j_k$
such that $\lambda_{k} = s_k\beta^{j_k}$.
Then for $s_k\beta^{j_k-1}$, it follows that the Armijo's condition \eqref{eq3} does not hold.
By Proposition \ref{lm5}(b), $f(x_k) \leq \mathcal{T}_k$ for all $k \geq 0$, we have
\begin{align}\label{eq3.2}
f(x_k + s_k\beta^{j_k-1}d_k)
&> \mathcal{T}_k + \delta_1s_k\beta^{j_k-1} \nabla f(x_k)^{\top}d_k
   - \delta_2 \left( s_k\beta^{j_k-1} \right)^2 \|d_k\|^2\notag \\
& \geq f(x_k) + \delta_1s_k\beta^{j_k-1}\nabla f(x_k)^{\top}d_k-\delta_2s_{k}^2\beta^{2(j_k-1)}\|d_k\|^2.%
\end{align}
On the other hand by the mean value theorem of integral calculus, $f(b) - f(a) = (b-a)\int_{a}^{b}f'(t)dt$, we evaluate
$$ f(x_k+h d_k) - f(x_k) = hd_k\int_{0}^{1}\left( \nabla f(x_k+th d_k)\right)^{\top}dt,$$
where $h > 0$.
This implies that
$$ f(x_k+h d_k) - f(x_k) = hd_k\int_{0}^{1}\left( \nabla f(x_k+th d_k) - f(x_k) + f(x_k)\right)^{\top}dt,$$
which gives,
$$ f(x_k+h d_k) - f(x_k)
=h\nabla f(x_k)^{\top}d_k + \int_{0}^{1}\left( \nabla f(x_k+th d_k) - \nabla f(x_k)\right)^{\top}h d_kdt.$$
Since $\nabla f$ is Lipschitz continuous, it follows that
\begin{align*}
f(x_k+h d_k) - f(x_k)
&\leq h \nabla f(x_k)^{\top}d_k + \int_{0}^{1}th^2L\|d_k\|^2 dt\\
&=h\nabla f(x_k)^{\top}d_k+\frac{h^2L\|d_k\|^2}{2}.
\end{align*}
This together with the inequality \eqref{eq3.2} by replacing $h = s_k\beta^{j_k-1}$, implies that
$$ s_k\beta^{j_k-1} \nabla f(x_k)^{\top} d_k  + \frac{s_k\beta^{2j_k-2}L\|d_k\|^2}{2}
\geq \delta_1s_k\beta^{j_k-1}\nabla f(x_k)^{\top}d_k-\delta_2s_{k}^2\beta^{2(j_k-1)}\|d_k\|^2,$$
that is,
$$ (1-\delta_1) \left( -\nabla f(x_k)^{\top}d_k \right) \leq (1+\delta_2L/2)s_{k}\beta^{j_k-1}\|d_k\|^2,$$
and so,
\begin{align*}
s_k\beta^{j_k} \geq \frac{(1-\delta_1)\beta}{(1+\delta_2 L/2)}\frac{-\nabla f(x_k)^{\top}d_k}{\|d_k\|^2}.
\end{align*}
By Proposition \ref{lm1}(a), we have
\begin{align*}
s_k\beta^{j_k} \geq \frac{(1-\delta_1)\beta}{(1+\delta_2 L/2)\mu\alpha_{\max}} = \lambda_{\min},
\end{align*}
that is,
$$ \lambda_{k} \geq \lambda_{\min}.$$

(b) Suppose there exists $c>0$ such that $\nabla f(x_k)^{\top}d_k \leq -c\|\nabla f(x_k)\|^2$.
From Cauchy-Schwarz inequality, we have
$$ \|\nabla f(x_k)\| \|d_k\| \geq -\nabla f(x_k)^{\top}d_k  \geq c \|\nabla f(x_k)\|^2,$$
that is,
$$ \|d_k\| \geq c \|\nabla f(x_k)\|.$$
Since the sequence $\{\lambda_{k}\}$ satisfies the Armijo condition \eqref{eq3},
we have $\lambda_{k} = s_{k} \beta^{j_{k}}$, and so,
$$ \lambda_{k} \leq s_k= -\frac{\nabla f(x_k)^{\top}d_k}{\|d_k\|^2}
\leq \frac{\|\nabla f(x_k)\|}{\|d_k\|} \leq \frac{1}{c}.$$
\qed
\end{proof}

\begin{theorem}\label{th1}
Let $K$ be a nonempty closed and convex subset of $\mathbb{R}^n$ and
$f : K\to \mathbb{R}$ be a differentiable function such that the assumptions (A1), (A2) and (A3) hold.
Let $\{x_k\}$ be a sequence generated by Algorithm {\rm\ref{A:1}}.
Assume that there exists $c>0$ such that direction assumption \eqref{eq:direction_assumption} holds.
Then the following statements hold:
\begin{itemize}
\item[{\rm(a)}] $f(z) \leq f(x_{k}) \leq \mathcal{T}_{k}$ for all $z\in\Omega_{\min}(f,K)$ and $k\ge 0$,
and consequently, the sequence $\{\mathcal{T}_k\}$ is convergent,
where $\{\mathcal{T}_k\}$ is defined by \eqref{eq7}.

\item[{\rm(b)}] $\ds\lim_{k\to\infty}\|\nabla f(x_k)\| = 0$ with the following error estimate:	
\begin{equation}\label{e8}
\Vert \nabla f(x_{k})\Vert ^{2}\leq
\frac{\mathcal{T}_{k}-\mathcal{T}_{k+1}}{\kappa (1-\eta _{\max })}, \quad \forall k\ge0,
\end{equation}
where
$$\kappa =\frac{c^{2}\zeta }{\mu ^{6}\alpha _{\max }^{2}}.$$

\item[{\rm(c)}] Every accumulation point of the sequence $\{x_k\}$
generated by Algorithm  {\rm\ref{A:1}} is a solution of VIP \eqref{Statpoint:eq1},
consequently, to a solution of the minimization problem \eqref{P1} provided $f$ is pseudo-convex.
\end{itemize}
\end{theorem}

\begin{proof}
(a)	Let $z\in\Omega_{\min}(f,K)$. 	
From Proposition \ref{lm2}(b), we have
\begin{equation}\label{eq2.11b}
f(z) \leq f(x_{k}) \leq \mathcal{T}_{k}, \quad \forall k\ge 0.
\end{equation}
Proposition \ref{lm2}(c) shows that the sequence $\{\mathcal{T}_k\}$ is monotonically non-increasing.
Thus, from \eqref{eq2.11b}, we see that the sequence $\{\mathcal{T}_k\}$ is convergent.

(b)	Let $z\in\Omega_{\min}(f,K)$.
From the Armijo's condition \eqref{eq3}, we have
\begin{eqnarray}\label{eq3.22}
f(x_{k+1})
&\le & \mathcal{T}_k + \delta_1\lambda_k \nabla f(x_k)^{\top}d_k-\delta_2\lambda_{k}^2\|d_k\|^2 \notag\\
&\leq& \mathcal{T}_k + \delta_1\lambda_k \nabla f(x_k)^{\top}d_k.
\end{eqnarray}
From \eqref{eq7} and \eqref{eq3.22}, we get
\begin{align}\label{eq2.11}
\mathcal{T}_{k+1}
&= \eta_{k+1}\mathcal{T}_k + (1-\eta_{k+1})f(x_{k+1})\notag\\
&\leq \eta_{k+1}\mathcal{T}_k + (1-\eta_{k+1}) \left( \mathcal{T}_k+\delta_1\lambda_k \nabla f(x_k)^{\top}d_k \right)
\notag\\
&=\mathcal{T}_k + \delta_1(1-\eta_{k+1})\lambda_k \nabla f(x_k)^{\top}d_k.
\end{align}
Since $\eta _{k}\leq \eta _{\max }<1$, we have $1-\eta _{\max} \leq 1-\eta _{k}$ and
$\frac{1}{1-\eta _{k}}\leq \frac{1}{1-\eta_{\max }}$. Therefore,
$$-\delta_1(1-\eta_{\max})\lambda_k \nabla f(x_k)^{\top}d_k
\leq -\delta_1(1-\eta_{k+1})\lambda_k \nabla f(x_k)^{\top}d_k \leq \mathcal{T}_k - \mathcal{T}_{k+1}.$$
Summing the above inequality from $k=0$ to $m$, we obtain
\begin{equation}\label{eq2.11a}
-\delta_1(1-\eta_{\max})\sum_{k=0}^{m}\lambda_k \nabla f(x_k)^{\top}d_k
\leq \sum_{k=0}^{m}(\mathcal{T}_k - \mathcal{T}_{k+1}) = \mathcal{T}_0 -\mathcal{T}_{m+1}.
\end{equation}
From \eqref{eq2.11b}, we have
\begin{equation}\label{eq2.11c}
f(z) \leq f(x_{m+1})  \leq \mathcal{T}_{m+1}.
\end{equation}
Thus, from \eqref{eq2.11a} and \eqref{eq2.11c}, we get
\begin{eqnarray*}
-\sum_{k=0}^{m}\lambda _{k}\nabla f(x_{k})^{\top }d_{k}
\leq \frac{1}{\delta _{1}(1-\eta _{\max})}\left( \mathcal{T}_{0}-f(z)\right),
\end{eqnarray*}
and therefore,
\begin{equation*}
-\sum_{k=0}^{\infty}\lambda_k \nabla f(x_k)^{\top}d_k < \infty.
\end{equation*}
Thus, by virtue of Proposition \ref{lm6}, we have $\lambda_{k} \geq \lambda_{\min}$,
and hence, the above inequality implies that
\begin{equation}\label{eq1.42}
\lim\limits_{k\to\infty}  \nabla f(x_k)^{\top}d_{k} =0.
\end{equation}
Since $\nabla f(x_{k})^{\top }d_{k}\leq -c\Vert \nabla f(x_{k})\Vert ^{2}$,
by Proposition \ref{lm1}(b), we obtain
\begin{equation*}
\left( \frac{\nabla f(x_{k})^{\top }d_{k}}{\Vert d_{k}\Vert }\right)^{2}
\geq \frac{c^{2}}{\mu ^{6}\alpha _{\max }^{2}}\Vert \nabla f(x_{k})\Vert^{2}.
\end{equation*}
By using \eqref{Eq1H1H2}, the last inequality becomes
\begin{equation}\label{DRSEestkapa}
f(x_{k+1})
\leq \mathcal{T}_{k}-\zeta \left( \frac{\nabla f(x_{k})^{\top}d_{k}}{\Vert d_{k}\Vert }\right) ^{2}
\leq \mathcal{T}_{k}-\frac{c^{2}\zeta}{\mu ^{6}\alpha _{\max }^{2}}\Vert \nabla f(x_{k})\Vert ^{2}.
\end{equation}
Note that $\kappa =\frac{c^{2}\zeta }{\mu ^{6}\alpha _{\max }^{2}}$.
From \eqref{eq7} and Proposition \ref{lm3}, we obtain
\begin{align}
\mathcal{T}_{k+1}
& =\eta _{k+1}\mathcal{T}_{k}+(1-\eta _{k+1})f(x_{k+1})\notag  \label{eq1.55} \\
& \leq \eta _{k+1}\mathcal{T}_{k}+(1-\eta _{k+1})
\left(\mathcal{T}_{k}-\kappa \Vert \nabla f(x_{k})\Vert ^{2}\right)   \notag \\
& =\mathcal{T}_{k}-\kappa (1-\eta _{k+1})\Vert \nabla f(x_{k})\Vert ^{2}\notag \\
& \leq \mathcal{T}_{k}-\kappa (1-\eta _{\max })\Vert \nabla f(x_{k})\Vert^{2},
\end{align}
which implies that
\begin{equation*}
\Vert \nabla f(x_{k})\Vert ^{2}\leq \frac{\mathcal{T}_{k}-\mathcal{T}_{k+1}}{\kappa (1-\eta _{\max })}.
\end{equation*}
Note that the sequence $\{\mathcal{T}_k\}$ is convergent by part (a),
therefore, we conclude that
$$ \lim_{k\to\infty}\|\nabla f(x_k)\| = 0.$$

(c)	Proposition \ref{lm5} shows that $x_k\in \mathcal{L}$ for $k\geq 0$, and therefore,
$\{x_k\}$ is bounded by the assumption (A1).
Then, every subsequence of the sequence $\{x_k\}$ converges.
Consider an arbitrary subsequence $\{x_{k_j}\}$ of $\{x_k\}$ converges to a point $x^*\in K$.
Then, $x^*$ is an accumulation point of the sequence $\{x_k\}$.
This together with Proposition \ref{lm7} and the relation \eqref{eq1.42}
implies that $x^*$ is a solution of VIP \eqref{Statpoint:eq1}.
\qed
\end{proof}

\begin{remark}
Note that in Theorem \ref{th1}, we have not assumed the convexity of $f$.
However, such an assumption is taken in \cite[Theorem 3.5]{YWH18}.
\end{remark}


\section{R-linear and sub-linear convergence}\label{S:4}	


Recall that a sequence $\{x_k\}$ is said to be R-linear convergent to a limit $x^*$
if there exists a constant $c \in (0, 1)$ such that
$$ \|x_{k+1} - x^*\| \leq c \|x_k - x^*\|, \quad \mbox{for sufficiently large } k \geq 0.$$
A sequence $\{x_k\}$ is said to converge with a $\mathcal{O}(1/k)$ rate of convergence to $x^*$ if
$$ \|x_k - x^*\| = \mathcal{O}\left(\frac{1}{k}\right),$$
which means there exist a constant $m > 0$ and a sufficiently large $k_0>0$ such that
$$ \|x_k - x^*\| \leq \frac{m}{k}, \quad \forall k \geq k_0.$$
It is worth to mention that the error of an iterative method decreases at a constant rate with each iteration
in R-linear convergence.
In contrast, the $\mathcal{O}(1/k)$ rate of convergence describes a convergence rate where the error decreases
proportionally to $1/k$, signifying a slower convergence rate than linear convergence.

\vspace{5pt}

We now establish linear convergence for the proposed modified
non-monotone line search Algorithm \ref{A:1} when $f$ is strongly quasiconvex.

\begin{theorem}\label{tm2}
Let $K$ be a nonempty closed and convex subset of $\mathbb{R}^n$ and
$f: K\to \mathbb{R}$ be a differentiable  strongly quasiconvex function with modulus $\gamma >0$.
Assume that $x^{*}$ is the unique solution of the minimization problem \eqref{P1}
and the assumptions {\rm (A1)--(A3)} hold.
Let $\{x_k\}$ be a sequence generated by Algorithm \ref{A:1}.
Let $\mu_k\in [1,\mu]$ and assume that there exists $0 < c < 1$ such that
direction assumption \eqref{eq:direction_assumption} holds, then there exists $\theta \in (0,\,1)$ such that
\begin{align}\label{eq1.46}
f(x_k) - f(x^*) \leq \theta^k(f(x_0)- f(x^*)), \quad\forall k>0.
\end{align}
\end{theorem}

\begin{proof}
Using the assumption A1, $\mathcal{L}$ is bounded.
On the other hand from Proposition \ref{lm5}(d), the sequence $\{x_k\}$ lies in $\mathcal{L}$,
and so, $\{x_k\}$ is bounded.
Therefore, from \eqref{eq1.55}, we have
\begin{align}\label{eq1.38}
f(x_{k+1})&\leq  \mathcal{T}_{k}- \kappa(1-\eta_{\max})\Vert \nabla f(x_{k})\Vert ^{2}.
\end{align}
Now, from Proposition \ref{lm1}(b) and \eqref{Eq:3.11}, we have
$$ \|x_{k+1} - x_k\| =\lambda_{k} \|d_k\| \leq \frac{\mu^3\alpha_{\max}}{c}\|\nabla f(x_k)\| =c_1\|\nabla f(x_k)\| ,$$
where $c_1=\frac{\mu^3\alpha_{\max}}{c}$. By the Lipschitz continuity of $\nabla f$, it yields that
$$ \|\nabla f(x_{k+1}) - \nabla f(x_{k})\| \leq L \|x_{k+1} - x_k\| \leq Lc_1\|\nabla f(x_k)\|,$$
that is,
\begin{align}\label{eq1.39}
	\|\nabla f(x_{k+1}) \|
	&=\|\nabla f(x_{k+1}) - \nabla f(x_{k}) + \nabla f(x_{k})\|\notag\\
	&\leq  \|\nabla f(x_{k+1}) - \nabla f(x_{k})\| + \|\nabla f(x_{k})\|\notag\\
	&\leq Lc_1\|\nabla f(x_k)\| + \|\nabla f(x_{k})\|\notag\\
	& = \left(1+ Lc_1\right)\|\nabla f(x_{k})\| :=b\|\nabla f(x_{k})\|,
\end{align}
where $b =  \left(1+ Lc_1\right)$.

Note that from \eqref{eq7} and \eqref{eq1.38}, we have
\begin{align}\label{eq1.40}
\mathcal{T}_{k+1} - f(x^*)
&= \eta_{k+1}\mathcal{T}_{k} + (1-\eta_{k+1})f(x_{k+1}) - f(x^*)\notag\\
& \leq   \eta_{k+1}(\mathcal{T}_{k} - f(x^*))+ (1-\eta_{k+1})(f(x_{k+1}) - f(x^*))\\
& \leq \eta_{k+1}(\mathcal{T}_{k} - f(x^*))+ (1-\eta_{k+1})(\mathcal{T}_{k}
  - \kappa(1-\eta_{\max}) \|\nabla f(x_k)\|^2 - f(x^*))\notag\\
& = \mathcal{T}_{k} - f(x^*) - (1-\eta_{k+1})\kappa(1-\eta_{\max}) \|\nabla f(x_k)\|^2\notag\\
&\label{eq1a}\leq \mathcal{T}_{k} - f(x^*) - (1-\eta_{\max})^2\kappa \|\nabla f(x_k)\|^2.
\end{align}
Now, we show that
\begin{equation}\label{eq1.43}
\mathcal{T}_{k+1} - f(x^*) \leq \theta(\mathcal{T}_{k} - f(x^*)),
\end{equation}
where
$$\theta =1-(1-\eta_{\max})^2b_1\kappa,\quad  b_1= \frac{1}{\bar{\kappa} + \frac{b^2 L}{\gamma^2}} < 1
\quad \mbox{ and }\quad\bar{\kappa}=(1-\eta_{\max})\kappa.$$
Since $\zeta < 1$, by Proposition \ref{lm3} and $c < 1$, we have
$$ \kappa = \frac{c^{2}\zeta }{\mu ^{6}\alpha _{\max }^{2}} <1.$$
This together with $\eta_{max} < 1$ implies that $0< \theta < 1$.

To prove \eqref{eq1.43}, we have two cases.

\noindent{\sc Case 1:} If $\|\nabla f(x_k)\|^2 \geq b_1(\mathcal{T}_{k} - f(x^*))$. Then from \eqref{eq1a}, we get
\begin{align}\label{eq1.41}
\mathcal{T}_{k+1} - f(x^*)
	&\leq \mathcal{T}_{k} - f(x^*) - b_1(1-\eta_{\max})^2\kappa(\mathcal{T}_{k} - f(x^*))\notag\\
	&=(1-b_1(1-\eta_{\max})^2\kappa)(\mathcal{T}_{k} - f(x^*))\\
	&=\theta (\mathcal{T}_{k} - f(x^*)).\notag
\end{align}

\noindent{\sc Case 2:} If $\|\nabla f(x_k)\|^2 < b_1(\mathcal{T}_{k} - f(x^*))$. Since $f$ is strongly quasiconvex, then by Remark \ref{rm2}(b), we have
$$ f(x_{k+1} )- f(x^*) \leq \textcolor{red}{\frac{L}{\gamma^2}}\|\nabla f(x_{k+1})\|^2.$$
This together with the inequality \eqref{eq1.39} implies that
$$ f(x_{k+1}) - f(x^*) \leq \frac{L}{\gamma^2} b^2\|\nabla f(x_k)\|^2,$$
and therefore,
$$ f(x_{k+1}) - f(x^*) \leq \frac{b^2b_1 L}{\gamma^2} (\mathcal{T}_{k} - f(x^*)).$$
Then by the inequality \eqref{eq1.40} and the above inequality, we obtain
\begin{align*}
\mathcal{T}_{k+1} - f(x^*)
    & \leq   \eta_{k+1}(\mathcal{T}_{k} - f(x^*))+ (1-\eta_{k+1})(f(x_{k+1}) - f(x^*))\\
	& \leq \eta_{k+1}(\mathcal{T}_{k} - f(x^*))+ (1-\eta_{k+1})\frac{b^2b_1 L}{\gamma^2}(\mathcal{T}_{k} - f(x^*))\\
	&= (\eta_{k+1} + \frac{b^2b_1 L}{\gamma^2}(1-\eta_{k+1}) )(\mathcal{T}_{k} - f(x^*))\\
	&=(\eta_{k+1} + (1-b_1\bar{\kappa})(1-\eta_{k+1}) )(\mathcal{T}_{k} - f(x^*))\\
	&\leq (1-(1-\eta_{\max}) b_1\bar{\kappa})(\mathcal{T}_{k} - f(x^*))\\
	&=(1-(1-\eta_{\max})^2 b_1\kappa)(\mathcal{T}_{k} - f(x^*))\\
	&=\theta(\mathcal{T}_{k} - f(x^*)).
\end{align*}
As $\mathcal{T}_{0} = f(x_0)$ and $f(x_k) \leq \mathcal{T}_k$,
from {\sc Case 1} and {\sc Case 2}, we obtain
\begin{align*}
		f(x_{k}) - f(x^{*}) & \leq \left( \mathcal{T}_{k} - f(x^{*}) \right)\\
		& \leq \theta \left( \mathcal{T}_{k-1} - f(x^{*}) \right)\\
		& \leq \theta^2 \left( \mathcal{T}_{k-2} - f(x^{*}) \right)\\
		&\,\,\,\vdots\\
		& \leq \theta^{k} \left( \mathcal{T}_{0} - f(x^{*}) \right)
		= \theta^{k} \left( f(x_{0}) - f(x^{*}) \right).
	\end{align*}
This completes the proof.
\qed
\end{proof}	

Theorem \ref{tm2} states that the sequence generated by Algorithm \ref{A:1}
is linearly convergent to a solution of the minimization problem \eqref{P1},
and consequently, to a solution of VIP \eqref{Statpoint:eq1} as $f$ is a strongly quasiconvex function.


\section{Numerical Applications}\label{S:5}


This section is devoted to the comparison of Algorithm \ref{A:1} (SGM, in short),
algorithm proposed in \cite{YWH18} (YWH, in short) and in \cite{ZH04} (ZH, in short).
For the computational experiments, we consider the following parameter settings:
\begin{itemize}
	\item [(i)]Step size: $\alpha_k = 1-\frac{1}{(n+1)^{0.5}}$
	\item [(ii)]Regularization parameter: $\lambda_0 =1$
	\item [(iii)]Decay parameters:  $\mu_k = 1+\frac{1}{(n+1)^2}$ and $\eta_k = 1-\frac{1}{(n+1)^{0.5}}$.

\end{itemize}
For SGM, we set the parameters $\mu = 1.25$, $\beta = 0.5$, $\delta_1= 0.001$ and $\delta_2=0.0001$.
In this setting, we evaluate the performance of SGM, YWH, and ZH
for Examples \ref{ex1}, \ref{ex2}, and \ref{ex3}.
In this study, we focus on analyzing the convergence behaviour of the objective function $f$
by examining the decrease in $\|f(x_{k+1}) - f(x^*)\|$ and $\|\nabla f(x_k)\|$ at each iteration.
All computations are implemented by using MATLAB version R2024b.

For both SGM and YWH, we employ the following scaled projection operator
to ensure that the iterates remain within the feasible set $ K $:
\begin{equation}\label{eq:projection}
P_{K, \mathcal{D}}(x) = \argmin_{y \in K} \left\{\|y - x\|_\mathcal{D}\right\}
= \argmin_{y \in K} \left\{ \frac{1}{2} y^{\top} \mathcal{D} y - y^{\top} \mathcal{D} x \right\},
\end{equation}
where $\mathcal{D}$ is the scaling matrix that adapts the geometry of the space.
This projection is computed with respect to the scaled norm
$\|y - x\|_\mathcal{D} = \sqrt{(y - x)^{\top} \mathcal{D} (y - x)}$.

On the other hand, for ZH, we use the standard Euclidean projection without the scaling matrix, i.e.,
$$ P_{K}(x) = \argmin_{y \in K} \left\{ \|y - x\|_2 \right\}.$$

Scaling matrices are helpful in gradient-based optimization algorithms,
particularly for high-dimensional and complex problems.
They adapt the algorithm to the optimized function's shape,
allowing it to handle areas with uneven curvature more efficiently.
By improving the problem's stability, scaling matrices balance out steep or flat regions, making updates more reliable.
They also enable the algorithm to take more effective steps, reducing the number of iterations
and speeding up convergence.
In cases with large-scale non-convex fractal programming and quadratic programming  or structured data,
scaling matrices focus the algorithm's attention on the critical, non-zero components, enhancing overall efficiency.

In SGM and YWH, the scaling matrix $\mathcal{D}$ is used in the projection step,
leading to faster convergence compare to ZH, where standard Euclidean projection is used.
The numerical experiments demonstrate that SGM converges faster,
with fewer iterations and lower CPU time, particularly in Example \ref{ex2} and \ref{ex3},
where the large-scale structure is pronounced.
Also, the modifications in \eqref{eq3} and \eqref{eq7} contribute to a more stable (non-oscillatory)
convergence behavior than in YWH.
This makes SGM effective for optimization in sparse and high-dimensional settings
(see the Figures \ref{fig1},  \ref{fig3} and \ref{fig:convergence}).
	
\subsection{Fractional programming}

Fractional programming is used in production planning, financial and corporate planning,
health care and hospital planning, and ratio problems;
this type of problem has garnered a lot of attention and research.
Numerous approaches to address these issues are given in \cite{BBT06}.	

\begin{example}\label{ex1}(See \cite{HM21})
Consider a feasible set
$$K  = \{ X\in \mathbb{R}^{5\times1} : -1\leq x_i\leq 1, i=1,2,\ldots, 5 \},$$
where $\mathbb{R}^{5\times 1}$ denotes the collection of $5 \times 1$ matrices,
and define a fractional function $f:\mathbb{R}^{5\times 1}\to \mathbb{R}$ by
$$ f(X) = \frac{X^{\top}WX + w_1^{\top}X +v_1}{w_2^{\top}X+v_2},$$
where $w_1 = (1, 2, -1, -2, 1)^{\top}\in\mathbb{R}^{5\times1}$,
$w_2 = (1, 0, -1, 0, 1)^{\top}\in\mathbb{R}^{5\times1}$,
$v_1=-2, \,v_2=20$ and $W$ is a symmetric positive definite matrix, given by
$$ W=\begin{pmatrix}
			5 & -1 &2&0&2\\
			-1 & 6 &-1&3&0\\
			2 & -1 &3&0&1\\
			0 & 3 &0&5&0\\
			2 & 0 &1&0&4
	\end{pmatrix}.$$
Note that $f$ is not convex on $K$, but it is pseudo-convex on $K$, that is,
$$\nabla f(X)^{\top} (Y-X) \geq 0 \, \Rightarrow \, f(X) \leq f(Y).$$
The gradient $\nabla f$ of $f$ is given by
$$ \nabla f(X) = \frac{(w_2^{\top}X+v_2)(2WX + w_1) - w_2(X^{\top}WX +w_1^{\top}X + v_1) }{(w_2^{\top}X+v_2)^{2}},$$
and $\nabla f$ is Lipschitz continuous on $K$ with the constant $L \approx 149$ (see \cite{HM21}). Moreover, the Hessian of $f$, given by
\[
\nabla^2 f(X)
\;=\;
\frac{2W}{w_2^{\top}X+v_2}
\;-\;
\frac{(2W X + w_1)\,w_2^\top + w_2\,(2W X + w_1)^\top}{(w_2^{\top}X+v_2)^2}
\;+\;
\frac{2(2W X + w_1)\,w_2\,w_2^\top}{(w_2^{\top}X+v_2)^3}.
\]


By using the fmin formula in Matlab, we obtain a solution
$f(X^*) = -0.158368$
of the minimization problem \eqref{P1}.
Moreover, we get $\|\nabla f(X^*)\| < 10^{-5}$.
Since $f$ is pseudo-convex, $X^*$ is a solution VIP \eqref{Statpoint:eq1}.
To further analyze the convergence behaviour of SGM, YWH and ZH,
we consider the initial point $ X_0 = ones(5,1)$ and
using the scaling matrix $\mathcal{D} = \nabla^2 f(X)$.
The results are summarized in Table~\ref{Table1}, where we compare the number of iterations (No. Itr.),
CPU time, and the error obtained for three methods: SGM, YWH, and ZH.
The results clearly show the efficiency of proposed method, especially in
terms of convergence speed and error minimization.		
		
\begin{table}[!htbp]
	\centering
	   \begin{minipage}{\textwidth}
			\rule{\textwidth}{1pt}
				\begin{tabular*}{\textwidth}{@{\extracolsep{\fill}}lccccccc@{\extracolsep{\fill}}}
					\multicolumn{1}{c}{Initial Point} & \multicolumn{3}{c}{$\|f(X_k) - f(X^*)\|$}
                    & \multicolumn{3}{c}{$\|\nabla f(X_k)\|$} \\
					\cline{2-4} \cline{5-7} \\
					$ X_0 = {\rm ones}(5,1)$ & SGM & YWH & ZH & SGM & YWH & ZH \\
					\hline
					& & & & & & \\
					No. Itr. & 31 & 51 & 83 & 44 & 100 & 100 \\
					& & & & & & \\
					CPU (s) & 1.11 & 1.9 & 2.52 & 1.01 & 1.39 & 1.96 \\
					& & & & & & \\
					Error & 3.43e-7 & 3.07e-7 & 3.8e-7 & 2.26e-5 & 1.14e-3 & 8.13e-5 \\
					& & & & & & \\
					\hline
			\end{tabular*}
		\end{minipage}
		  \caption{Convergence Results of Example \ref{ex1}}\label{Table1}
\end{table}

The convergence analysis is further demonstrated graphically in Figure~\ref{fig1}.
Subfigure (a) shows the results for $\|f(X_k) - f(X^*)\|$,
and Subfigure (b) shows the results for $\|\nabla f(X_k)\|$, with the corresponding data presented in Table~\ref{Table1}.
		

\begin{figure}
	\centering
	\subfloat[Convergence Analysis of $\|f(X_k) - f(X^*)\|$ results in Table \ref{Table1}]{%
		\resizebox*{5cm}{!}{\includegraphics{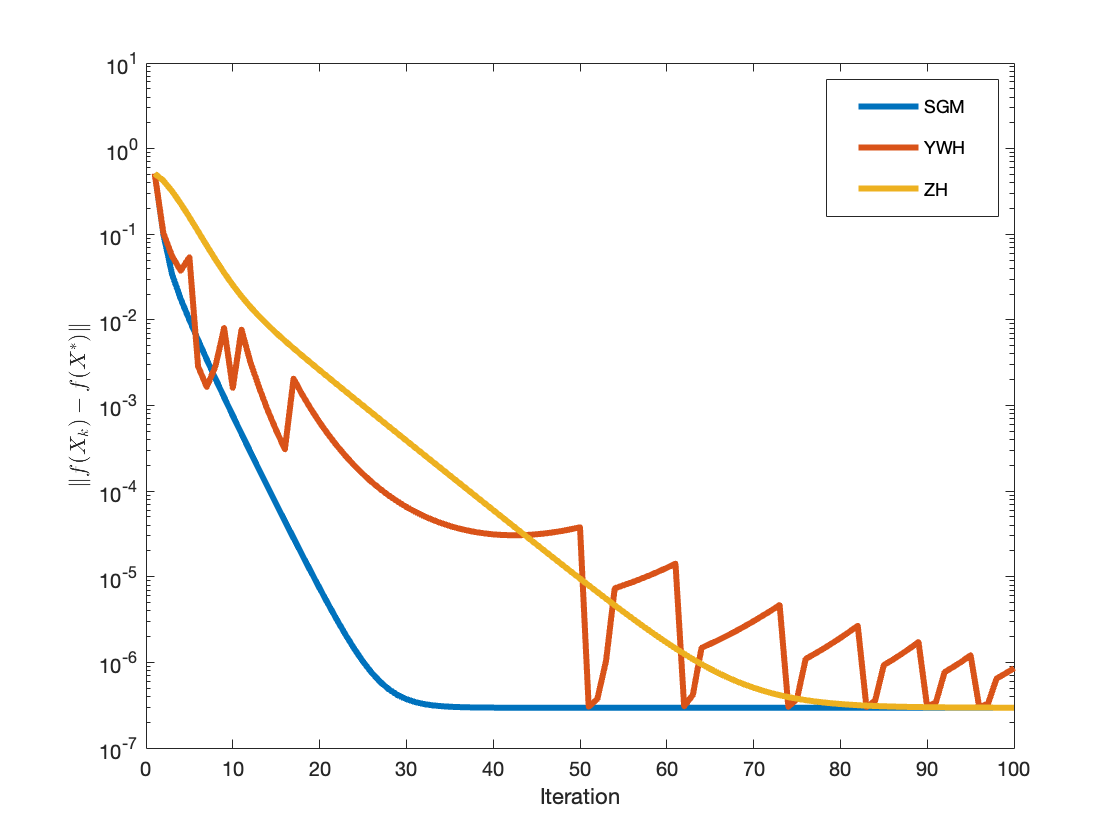}}}\hspace{5pt}
	\subfloat[Convergence Analysis of $\|\nabla f(X_k)\|$ results in Table \ref{Table1}]{%
		\resizebox*{5cm}{!}{\includegraphics{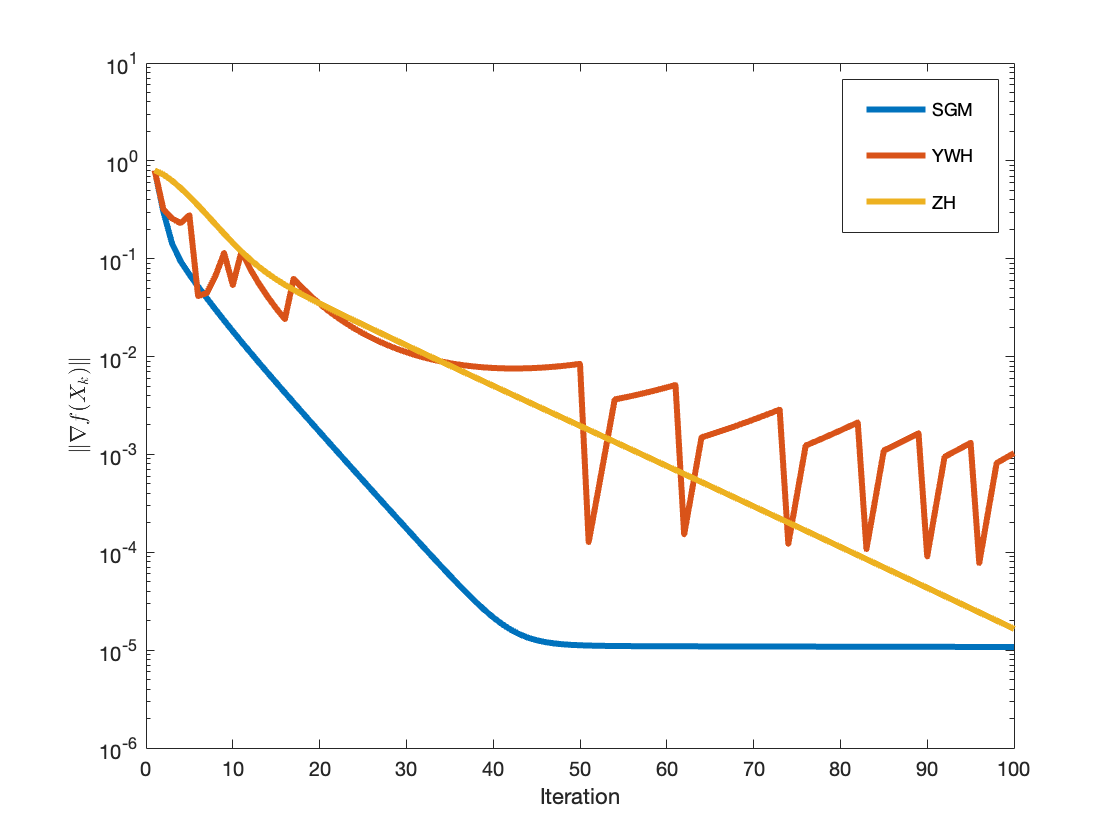}}}
	\caption{Convergence Analysis of Example \ref{ex1}} \label{fig1}
\end{figure}
		
In conclusion, SGM demonstrates superior convergence behaviour,
significantly when scaling matrix $ \mathcal{D} $ is used.
The numerical and graphical results confirm the effectiveness of
our approach the objective function is non-convex.
\end{example}		

\subsection{Large-scale fractional programming}				
	
\begin{example}\label{ex2}
	Let \(A,B \in \mathbb{R}^{n\times n}\) be two given matrices and let \(a,b \in \mathbb{R}^n\), \(\alpha,\beta \in \mathbb{R}\) be fixed parameters.
	Assume that \(A\) is symmetric and positive definite, and let \(\lambda_{\min}(A)>0\) denote its smallest eigenvalue.
	Consider two constants \( m < M\), and define the feasible set
	\[
	K = \Big\{ X \in \mathbb{R}^n : m \le \tfrac{1}{2}X^\top B X + b^\top X + \beta \le M \Big\}.
	\]
	Suppose \(B\) is positive semidefinite  and
	\(\tfrac{1}{2}X^\top A X + a^\top X + \alpha \le 0\) for every \(X \in K\).
	Then, the function \(f: K \to \mathbb{R}\) defined by
	\begin{equation}\label{LFP}
	f(x) =
	\frac{
		\tfrac{1}{2}X^\top A X + a^\top X + \alpha
	}{
		\tfrac{1}{2}X^\top B X + b^\top X + \beta
	},\quad \forall X\in \mathbb{R}^{n \times 1}
	\end{equation}
	is strongly quasiconvex (see, \cite[Corrolary 4.1]{ILMY24}) on \(K\) with modulus
	\(
	\gamma' = \frac{\lambda_{\min}(A)}{M} > 0.
	\)
	Let
	\[
	f(x)=\frac{p(x)}{q(x)},\qquad
	p(x)=\tfrac{1}{2}x^\top A x + a^\top x + \alpha,\qquad
	q(x)=\tfrac{1}{2}x^\top B x + b^\top x + \beta.
	\]
	Set $\nabla q(x)=A x + a$, $ \nabla d(x)=B x + b$, $\nabla^2 q(x)=A$ and $\nabla^2 d(x)=B$.	Gradient of $f$, given by
	\[
	\nabla f(x)=\frac{d(x)\,\nabla q(x)-q(x)\,\nabla d(x)}{[d(x)]^{2}}
	\]
	Moreover, Hessian of $f$, given by
	\[
	\nabla^2 f(x)
	=\frac{d(x)\,\nabla^2 q(x)-q(x)\,\nabla^2 d(x)}{d(x)^2}
	-\frac{\,g(x) \nabla d(x)^{\top}\;+\;\nabla d(x) g(x)^{\top}\,}{d(x)^2}\,.
	\]
	where
	\[
	g(x)\;:=\;\nabla q(x) - \frac{q(x)}{d(x)}\,\nabla d(x).
	\]
	
	Consider the large-scale  fractional programming problem defined in \eqref{LFP} for $n=512$ with the setups of the data matrix $A, B \in \mathbb{R}^{n \times n} $ defined by $randn(n,n)$, $a = 0.1 * randn(n,1)$, $b = 0.05 * randn(n,1)$, $\alpha = 1.0$, $\beta  = 100$, $m=-1$ and $M=1$.
	We also define the scaling matrix $\mathcal{D}=\nabla^2 f(x)$.
	
	It follows from fmin formula in Matlab, we obtain a solution $f(X^*) = -3.347045e-02$
	of the minimization problem \eqref{P1}, and $\|\nabla f(X^*)\| < 10^{-5}$.
	Since $f$ is convex, $X^*$ is also a solution of VIP \eqref{Statpoint:eq1}.
	We solve the sparse problem \eqref{LFP} for each combination of algorithm and scaling matrix,
	and the results are presented in Table~\ref{Table3} and Figure~\ref{fig3}.
	\begin{table}[!htbp]
		\begin{center}
			\begin{minipage}{\textwidth}
				\rule{\textwidth}{1pt}
				\begin{tabular*}{\textwidth}{@{\extracolsep{\fill}}lccccccc@{\extracolsep{\fill}}}\\
					\multicolumn{1}{c}{Initial point} &\multicolumn{3}{c}{$\|f(X_k) - f(X^*)\|$ }  &\multicolumn{3}{c}{$\|\nabla f(X_k)\|$}\\
					\cline{2-4} \cline{5 - 7}\\
					$ X_0 = {\rm ones}(n,1)$ & SGM&YWH& ZH&SGM&YWH& ZH  \\
					\hline
					&& &&         &&        \\
					No.Itr. & 500&500&500     &500&500  &500         \\
					&& &&           &&        \\
					CPU(s) &1.53& 2.18&2.09&  1.14  & 1.89& 2.03         \\
					&& &&           &&        \\
					Error & 9.57e-8&2.2e-3&1.4e-1     &7.7e-3&7.2e-2  &2.9e-1        \\
					&& &&           &&    \\
					\hline
				\end{tabular*}
			\end{minipage}
			\vspace{\abovecaptionskip}
			\caption{Convergence table of the Example \ref{ex2}}\label{Table3}
		\end{center}
	\end{table}	
\begin{figure}
	\centering
	\subfloat[Convergence Analysis of $\|f(X_k) - f(X^*)\|$ in Example \ref{ex2} and results shown in Table \ref{Table3}]{%
		\resizebox*{5cm}{!}{\includegraphics{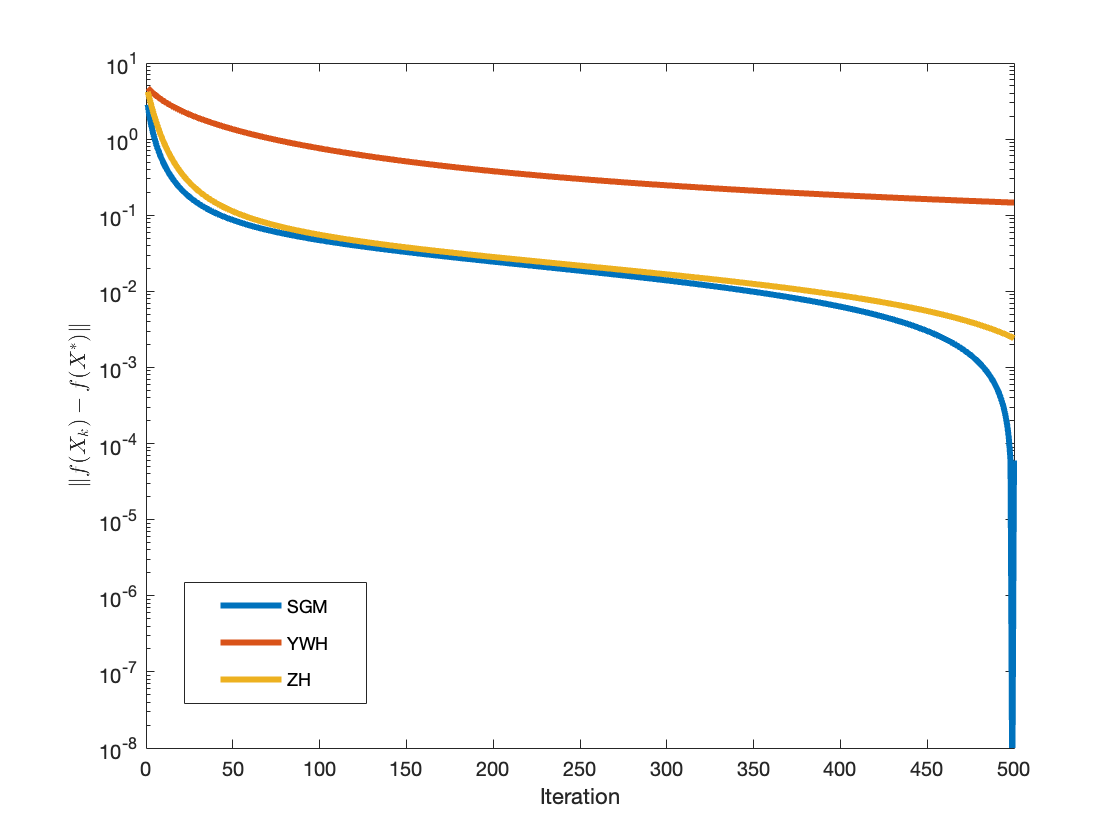}}}\hspace{5pt}
	\subfloat[Convergence Analysis of $\|\nabla f(X_k)\|$ in Example \ref{ex2} and results shown in Table \ref{Table3}]{%
		\resizebox*{5cm}{!}{\includegraphics{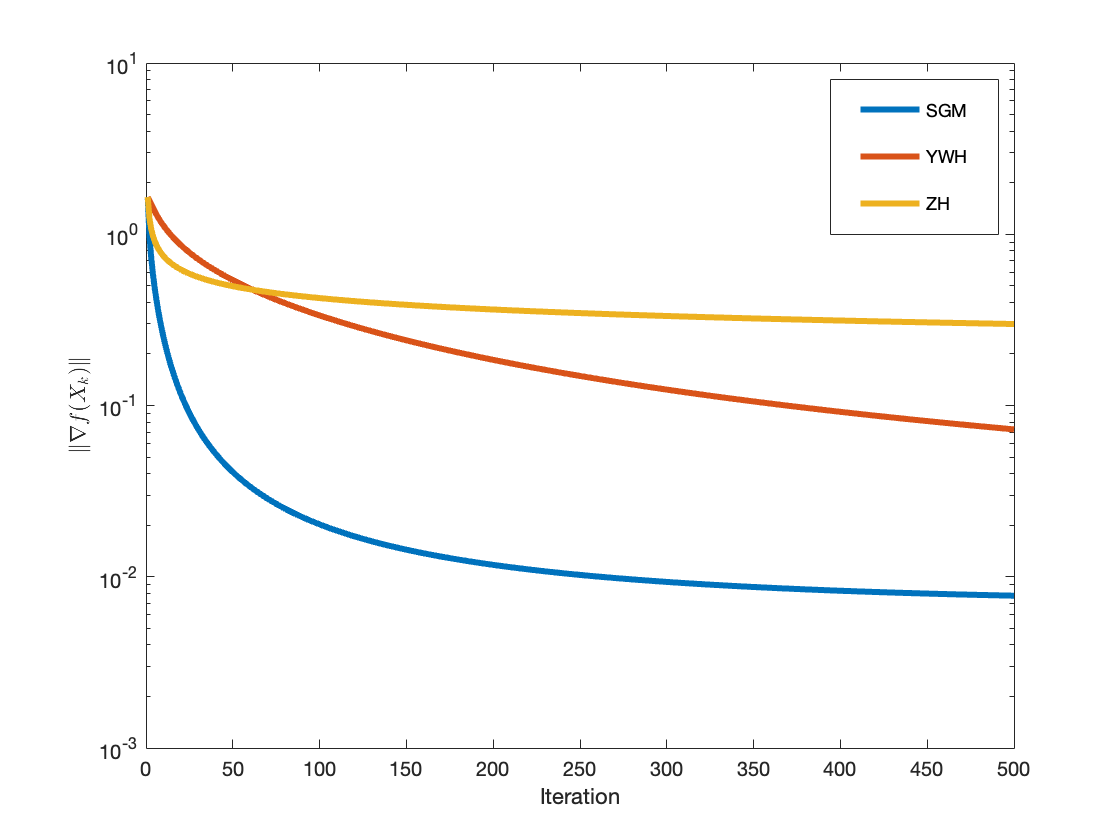}}}
	\caption{Convergence Analysis of Example \ref{ex2}} \label{fig3}
\end{figure}

\end{example}
	
\subsection{Quadratic programming}

Quadratic programming is to optimize a quadratic objective function with linear constraints.
It is widely used in various fields such as portfolio optimization in finance,
parameter estimation in machine learning,
optimal control in engineering and resource allocation in operations research, see \cite{BV04}.
The versatility and effectiveness of quadratic programming makes it valuable
for tackling complex optimization problems with a quadratic objective function and linear constraints.

\begin{example}\label{ex3}
Consider a feasible set $K  = [-1,1]^n$ and
a quadratic function $f : \mathbb{R}^{n\times 1}\to \mathbb{R}$ defined by
$$ f(X) = X^\top V X - p^\top V p + p^\top W(X-p), \quad \forall X\in \mathbb{R}^{n \times 1},$$
where $V$ and $W$ are  symmetric positive definite matrices and $p\in \mathbb{R}^{n \times1}$. Consider $n=256$, then it can be easily to find the smallest eigen value of $V$ is $\gamma = 1.4942\times 10^{-4}$. Note that $f$ is strongly quasiconvex with modulus $\gamma$ (see, \cite[Example 4.1]{IL22}).
		
Now for numerical viewpoint, we consider $p = (1, 0, \ldots, 0)^{\top} \in \mathbb{R}^{n \times 1}$
and the symmetric positive definite matrices
$$
V=\begin{pmatrix}
	2 & 1 & 0 & \cdots & 0\\[4pt]
	1 & 2 & 1 & \ddots & \vdots\\[4pt]
	0 & 1 & 2 & \ddots & 0\\[4pt]
	\vdots & \ddots & \ddots & \ddots & 1\\[4pt]
	0 & \cdots & 0 & 1 & 2
\end{pmatrix}\in\mathbb{R}^{n\times n},\qquad
W=\begin{pmatrix}
	3 & 0.5 & 0 & \cdots & 0\\[4pt]
	0.5 & 3 & 0.5 & \ddots & \vdots\\[4pt]
	0 & 0.5 & 3 & \ddots & 0\\[4pt]
	\vdots & \ddots & \ddots & \ddots & 0.5\\[4pt]
	0 & \cdots & 0 & 0.5 & 3
\end{pmatrix}\in\mathbb{R}^{n\times n}.
$$
The gradient $\nabla f$ and  Hessian $\nabla^2 f$ of $f$ are given by
$$ \nabla f(X) = 2VX+Wp \mbox{ and } \nabla^2 f(X) = 2V,\quad \forall X\in \mathbb{R}^{n \times 1}.$$

It can be easily seen that $\nabla f$ is Lipschitz continuous on $K$ with Lipschitz constant $L=7.99$.	
We use the  Hessian $\nabla^2 f(X)$ as a scaling matrix $\mathcal{D}$ for the numerical analysis.
By using the fmin formula in Matlab, we obtain $f(X^*) = -6.16771$ a solution
of the minimization problem \eqref{P1}.
Moreover, $\|\nabla f(X^*)\| < \varepsilon=10^{-6}$.
Since $f$ is convex, $X^*$ is a solution of VIP \eqref{Statpoint:eq1}.

We implement SGM, YWH, and ZH with the initial point $ X_0 = (0.8*{\rm ones}(5,1)$,
and compare their performance for $\|f(X_k) - f(X^*)\|$ and $\|\nabla f(X_k) \|$
with the scaling matrix $ \mathcal{D} = 2V$.
The convergence results are shown in Table~\ref{Table2} and illustrated in Figure~\ref{fig:convergence}.
		
\begin{table}[!htbp]
			\centering
			\begin{minipage}{\textwidth}
				\rule{\textwidth}{1pt}
				\begin{tabular*}{\textwidth}{@{\extracolsep{\fill}}lccccccc@{\extracolsep{\fill}}}
					\multicolumn{1}{c}{Initial Point} & \multicolumn{3}{c}{$\|f(X_k) - f(X^*)\|$}
                    & \multicolumn{3}{c}{$\|\nabla f(X_k)\|$} \\
					\cline{2-4} \cline{5-7} \\
					$ X_0 = {\rm ones}(n,1)$ & SGM & YWH & ZH & SGM & YWH & ZH \\
					\hline
					& & & & & & \\
					No. Itr. & 27 & 54 & 100 & 45 & 64 & 100 \\
					& & & & & & \\
					CPU (s) & 0.68 & 4.55 & 3.58 & 0.75 & 4.87 & 5.57 \\
					& & & & & & \\
					Error & 3.5e-6 & 3.4e-6 & 1.4e-3 & 1.6e-6 & 1.4e-6 & 1.7e-2 \\
					& & & & & & \\
					\hline
				\end{tabular*}
			\end{minipage}
			\caption{Convergence results of Example \ref{ex3}}
			\label{Table2}
\end{table}
		

\begin{figure}
	\centering
	\subfloat[Convergence Analysis of $\|f(X_k) - f(X^*)\|$ in Example \ref{ex3} and results shown in Table \ref{Table2}]{%
		\resizebox*{5cm}{!}{\includegraphics{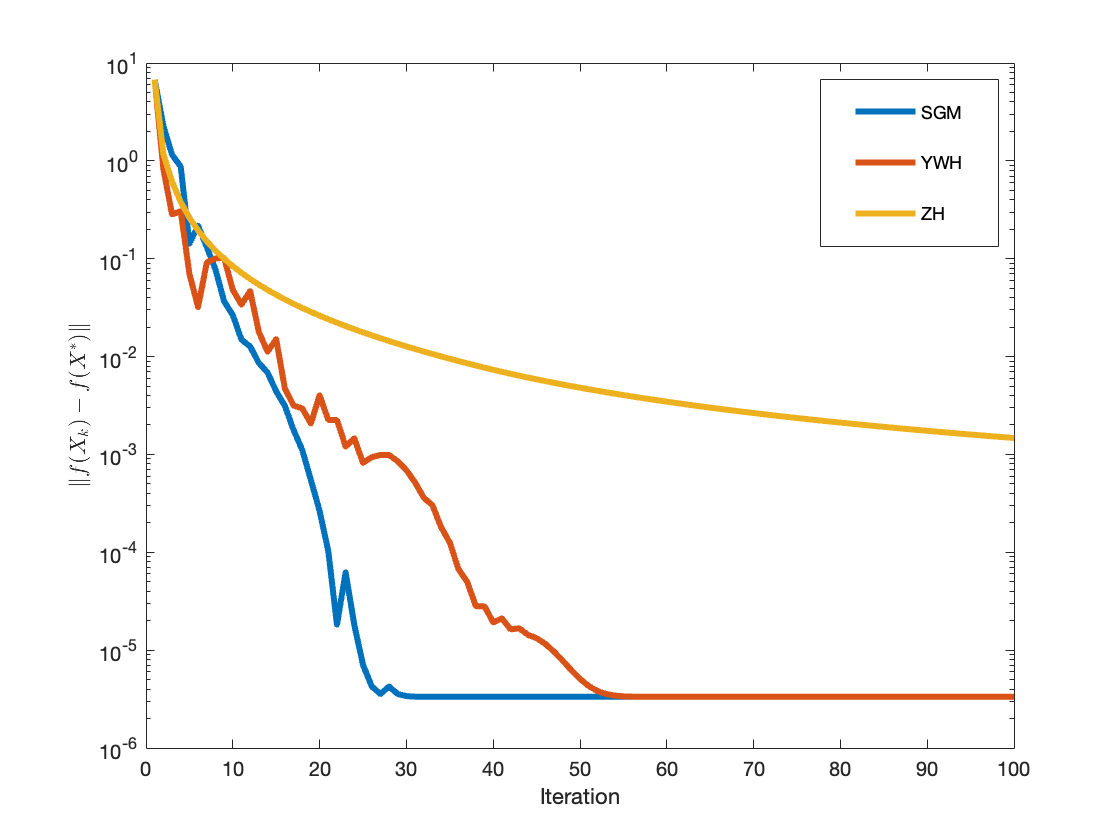}}}\hspace{5pt}
	\subfloat[Convergence Analysis of $\|\nabla f(X_k)\|$ in Example \ref{ex3} and results shown in Table \ref{Table2}]{%
		\resizebox*{5cm}{!}{\includegraphics{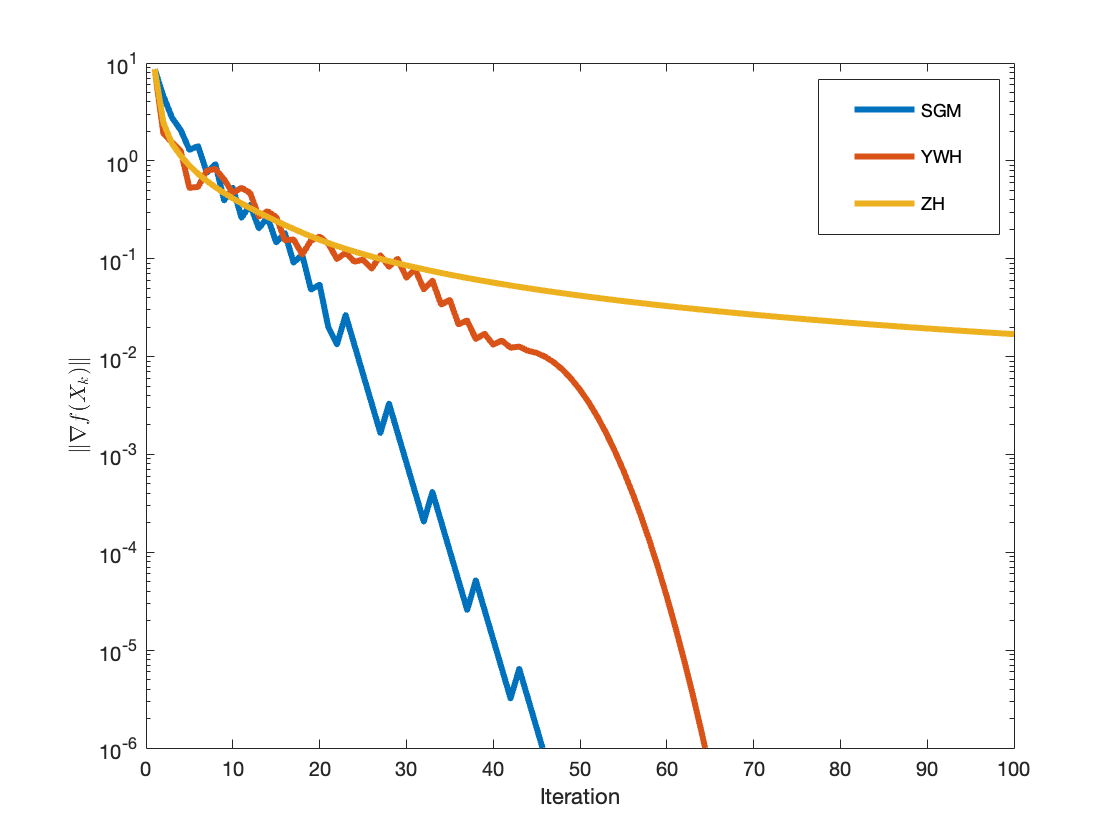}}}
	\caption{Convergence Analysis of Example \ref{ex3}} \label{fig:convergence}
\end{figure}
		
In conclusion, SGM demonstrates effective convergence behaviour
with a scaling matrix to optimize a convex objective function.
As presented in Table~\ref{Table2}, the numerical results confirm the efficiency of SGM
with the scaling matrix $\mathcal{D}$.

As demonstrated in Table~\ref{Table2}, SGM consistently outperforms YWH and ZH regarding convergence speed, computational efficiency and accuracy.
In both Cases, SGM converges in fewer iterations, requiring only 27 and 45 iterations, respectively, compared to the significantly higher iteration counts for YWH and ZH.

Regarding CPU time, SGM is the most efficient, taking just 0.68 seconds for $\|f(X_k) - f(X^*)\|$ and 0.75
seconds $\|\nabla f(X_k)\|$, whereas YWH and ZH require up to 5.57 seconds to converge.
These results confirm that SGM is faster and more accurate,
making it a highly efficient method for solving large-scale non-convex quadratic problems.
\end{example}

\section{Conclusion}

In this study, we proposed a scaled gradient modified non-monotone line search algorithm
for solving constrained minimization problems of the form \eqref{P1}
and examined its convergence behaviour.
The proposed algorithm integrates a scaling matrix to improve the efficiency of gradient based methods,
achieving faster convergence and higher accuracy.
Our analysis shows that incorporating the scaling matrix in the projection step enhances
the algorithm's adaptability to problem geometry, leading to fewer iterations
and reduce computational time compare to existing methods such as YWH and ZH.
Specifically, the proposed method exhibits superior performance in high-dimensional problems
where the curvature of the objective function varies significantly.
Moreover, numerical experiments demonstrated that the proposed algorithm consistently outperforms
the YWH and ZH methods in both convergence speed and error reduction,
particularly, in solving large-scale fractional and quadratic problems.
By leveraging the scaling matrix, our method achieves a linear convergence rate for strongly quasiconvex functions,
surpassing the slower $\mathcal{O}(1/k)$ convergence rates observed in the traditional methods.
Overall, this research highlights the benefits of scaling matrices in non-monotone line search methods,
offering significant improvements in computational efficiency, accuracy, and convergence for large-scale,
high-dimensional optimization problems.

\section*{Acknowledgment}
The research part of the first and second author was done during their visit to
the Center for General Education, China Medical University, Taichung 40402, Taiwan.
				
\section*{Funding}
No funding agency.

\section*{Availability of data and materials}
Not applicable.

\section*{Declarations}

The authors declare no competing interests.


\end{document}